\documentclass[12pt]{article}

\usepackage{amsfonts}
\usepackage{xr}
\usepackage{graphicx,psfrag}
\usepackage{amsmath,amsthm}
\usepackage{color,soul}
\usepackage[active]{srcltx}

\def\2{C^{1,2}(\R\times\R^N)}
\def\to{\rightarrow}
\def\e{\varepsilon}

\def\R{\mathbb{R}}

\def\tilde{\widetilde}
\def\.{\cdot}

\newcommand{\be}{\begin{equation}}

\newcommand{\ee}{\end{equation}}
\newcommand{\baa}{\begin{array}}
\newcommand{\eaa}{\end{array}}
\newcommand{\ba}{\begin{eqnarray}}
\newcommand{\ea}{\end{eqnarray}}

\newtheorem{thm}{\bf Theorem}[section]
\newtheorem{lem}[thm]{\bf Lemma}
\newtheorem{prop}[thm]{\bf Proposition}
\newtheorem{cor}[thm]{\bf Corollary}
\theoremstyle{definition}

\newenvironment{formula}[1]{\begin{equation}\label{eq:#1}}
                       {\end{equation}\noindent}
\def\Fi#1{\begin{formula}{#1}}
\def\Ff{\end{formula}\noindent}

\begin{document}
\date{}
\title{\bf{Gaussian estimates for general parabolic operators in dimension $1$}}
\author{Gr\'egoire Nadin\footnote{Institut Denis Poisson, Universit\'e d'Orl\'eans, Universit\'e de Tours, CNRS, Orl\'eans, France ({\tt gregoire.nadin@cnrs.fr}).}}

\maketitle

\begin{abstract} 
We derive in this paper Gaussian estimates for a general parabolic equation $u_{t}-\big(a(x)u_{x}\big)_x= r(x)u$ over $\R$. Here $a$ and $r$ are only assumed to be bounded, measurable and $\mathrm{essinf}_\R a>0$. We first consider a canonical equation $\nu (x) \partial_{t}p - \partial_{x }\big( \nu (x)a(x)\partial_{x}p\big)+W\partial_{x}p=0$, with $W\in \R$, $\nu$ bounded and $\mathrm{essinf}_\R \nu>0$, for which we derive Gaussian estimates for the fundamental solution:
$$\forall t>0, x,y\in \R, \quad \displaystyle\frac{1}{Ct^{1/2}}e^{-C|T(x)-T(y)-Wt|^{2}/t} \leq P(t,x,y)\leq \frac{C}{t^{1/2}}e^{-|T(x)-T(y)-Wt|^{2}/Ct}.$$
Here, the function $T$ is a corrector, for which we are able to derive appropriate properties using one-dimensional arguments. We then show that any solution $u$ of the original equation could be divided by some generalized principal eigenfunction $\phi_\gamma$ so that $p:=u/\phi_\gamma$ satisfies a canonical equation. As a byproduct of our proof, we derive Nash type estimates, that is, Holder continuity in $x$, for the solutions of the canonical equation. 
\end{abstract}

\noindent {\bf Key-words:} Gaussian estimates, parabolic equation, corrector, generalized eigenfunctions, Nash type estimates. 

\smallskip

\noindent {\bf AMS classification:} 35B40; 35K10; 35K15

%



\section{Introduction and main result}

\subsection{State of the art on Gaussian estimates}

In this paper, we consider the equation
\begin{equation} \label{eq:parab}\left\{\begin{array}{ll}
u_{t}-\big(a(x)u_{x}\big)_x= r(x)u \quad &\hbox{ for all } t>0, \ x\in \R,\\
u(0,x)=u_{0}(x) \quad &\hbox{ for all } x\in \R,\\
\end{array}\right. \end{equation}
where $u_{0}\in L^\infty (\R)$ is a compactly supported, nonnegative and non-null initial datum. 

We make the following hypotheses on the measurable functions $r$ and $a$ along all this article:
\begin{equation}\label{hyp}
\exists \mu>0, \quad |r(x)|\leq \mu, \quad \frac{1}{\mu}\leq a(x)\leq \mu \quad \hbox{ for a.e. } x\in \R. 
\end{equation}
Under these hypotheses, it is well-known that one can define a fundamental solution $U(t,x,y)$ of (\ref{eq:parab}), associated with the initial datum $\delta_y$. 

Of course a drift term $b(x)u_x$ could also be addressed, with $b$ bounded over $\R$. In that case, one uses the change of variables $v(t,x):=u(t,x)e^{\int_0^x \frac{b}{2a}}$ in order to reduce to an equation with no drift like (\ref{eq:parab}). 

\bigskip

When $r\equiv 0 $ and $a\equiv 1$, it is well-known that $U(t,x,y)\equiv \frac{1}{\sqrt{4\pi t}}e^{-|x-y|^{2}/4t}$. When $r\equiv 0$ and $a$ depends on $x$, 
it was proved by Aronson \cite{Aronson} using the Harnack inequality and by Fabes and Stroock \cite{FabesStroock} that there exists a constant $C>0$, which only depends on $\mu$, such that 
$$\frac{1}{Ct^{1/2}} e^{-C|x-y|^{2}/t} \leq U(t,x,y)\leq \frac{C}{t^{1/2}} e^{-|x-y|^{2}/Ct}.$$
Fabes and Stroock provided a direct proof of this estimate relying on methods developed by Nash \cite{Nash} to investigate the Holder continuity of the solutions. 

Several generalization of this result have been provided for bounded domains with Dirichlet \cite{ZhangDir}, Neumann \cite{Wang} or general \cite{Arendt, Daners} boundary conditions. We also refer to \cite{Kumagai} for Gaussian estimates on graphs, and to \cite{Grigor'yan97, Grigor'yan06} for Gaussian estimates on general manifolds. 

\bigskip

When $r\not\equiv 0$, but $r$ decays like $r(x)\simeq \frac{a}{1+|x|^b}$ at infinity, Zhang \cite{Zhang} derived Gaussian estimates (in the more general framework of Riemannian manifolds). 

If we do not impose a decay at infinity on $r$, then an exponential growth rate is expected in the estimate. Of course we could always easily derive from \cite{Aronson, FabesStroock} the estimate 
$$\frac{1}{Ct^{1/2}} e^{-C|x-y|^{2}/t- \|r\|_\infty t} \leq U(t,x,y)\leq \frac{C}{t^{1/2}} e^{-|x-y|^{2}/Ct+ \|r\|_\infty t}.$$
The aim of the present paper is to obtain more accurate estimates that encapsulate more precisely the exponential growth rate created by $r$.

In order to be more precise,  it is  convenient to investigate the canonical form of equation (\ref{eq:parab}), namely
\begin{equation} \label{eq:canonical}\left\{ \begin{array}{ll}
\nu (x) \partial_{t}p - \partial_{x }\big( \nu (x)a(x)\partial_{x}p\big)+W\partial_{x}p=0 &\hbox{ for all } t\in (0,\infty), x\in \R,\\
p(0,x) = p_0(x) &\hbox{ for all } x\in \R,\\
\end{array}\right. \end{equation} 
where $W$ is a constant and $\nu \in L^\infty (\R)$, $\mathrm{essinf}_\R \nu>0$. 
Such a canonical form arises when one divides the solution of  (\ref{eq:parab}) by some appropriate eigenfunction (see Section \ref{sec:UP} below), that is: $u(t,x)= \phi_\gamma (x) e^{\gamma t} p(t,x)$.

If $\nu -1$ admits a bounded primitive, then Norris \cite{Norris} proved that there exists a positive constant $C>0$ such that the fundamental solution $P(t,x,y)$ associated with equation (\ref{eq:canonical}) with initial datum $P(0,\cdot,y)=\delta_y /\nu$ satisfies
\begin{equation} \label{eq:Norris} \frac{1}{Ct^{1/2}} e^{-C|x-Wt-y|^{2}/t} \leq 
P(t,x,y)\leq \frac{C}{t^{1/2}} e^{-|x-Wt-y|^{2}/Ct}\end{equation}
and thus 
$$\frac{\phi_\gamma (x)}{Ct^{1/2}\phi_\gamma (y)} e^{-C|x-Wt-y|^{2}/t+\gamma t} \leq U(t,x,y)\leq \frac{C\phi_\gamma (x)}{t^{1/2}\phi_\gamma (y)} e^{-|x-Wt-y|^{2}/Ct+\gamma t}.$$
This is the type of estimates we want to derive when $r\not \equiv 0$, for a general dependence of $a$ and $r$ with respect to $x$.

The boundedness hypothesis on the primitive of $\nu-1$ is satisfied in particular when $a$ and $r$ are periodic, and in that case Norris even obtains more accurate estimates \cite{Norris}. But we could construct counter-examples for which it is not satisfied anymore, when $r$ is almost periodic for example. 

Lastly, let us mention a paper of Norris and Stroock \cite{NorrisStroock}, which addresses general $a$ and $r$. They obtain a lower bound and an upper bound that are not of the same type. It is not clear to us how to compare their result with the one derived in the present paper. 

\subsection{Statement of the result for the canonical equation}\label{sec:statement}

We address in this section the canonical equation (\ref{eq:canonical}) and we make the following hypotheses on the measurable functions $\nu$ and $a$:

\begin{equation}\label{hyp:original}
\exists\mu>0 \hbox{ s.t. } \quad \frac{1}{\mu}\leq \nu(x)\leq \mu, \quad \frac{1}{\mu}\leq a(x)\leq \mu \quad \hbox{ for a.e. } x\in \R. 
\end{equation}

\begin{thm}\label{thm:heatkernel} Assume (\ref{hyp:original}). Let $P(t,x,y)$ be the fundamental solution of (\ref{eq:canonical}), associated with the initial datum $\delta_y /\nu$.  Then there exists a constant $C>0$ (only depending on $\mu$, not on $W$) such that 
\begin{equation} \label{eq:lbub}
\forall t>0, x,y\in \R, \quad \displaystyle\frac{1}{Ct^{1/2}}e^{-C|T(x)-T(y)-Wt|^{2}/t} \leq P(t,x,y)\leq \frac{C}{t^{1/2}}e^{-|T(x)-T(y)-Wt|^{2}/Ct}. \end{equation}
\end{thm}

The function $T$ is defined by the following Proposition.

\begin{prop}\label{prop:T}
There exists a unique solution $T$ of 
\begin{equation} \label{eq:T}
-\big( a(x)\nu(x)T'\big)'+WT'=W\nu(x) \hbox{ in } \R, \quad T(0)=0
\end{equation}
such that $T(x)/|x|$ is bounded over $\R$. Moreover, one has $m\leq T'(x)\leq M$ for all $x\in \R$, for $m=1/\mu^5$ and $M=\mu^5$. 
\end{prop}

Assume that $\nu$ and $r$ are $1-$periodic in $x$, and that $\int_0^1 \nu = 1$ by rescaling. Then there exists a unique periodic solution $\chi=\chi(x)$ of 
$$-\big( a(x)\nu(x)\chi'\big)'+W\chi'=W\big(\nu(x)-1\big) +\big( a(x)\nu(x)\big)'\hbox{ in } \R, \quad \chi(0)=0$$
since the right-hand is of average $0$ (such a quantity is introduced in \cite{HNRR2} for example). One has $T(x)=\chi(x)+x$, with $\chi$ bounded. Hence, even if it means increasing $C$, one can recover Norris' estimate (\ref{eq:Norris}) with Theorem \ref{thm:heatkernel}.

For more general dependences in $x$, say $\nu$ and $a$ almost periodic for example, $\nu$ having average $1$, one can still define $\chi(x):= T(x)-x$, but this quantity is not bounded in $x$ in general. 

We could identify $\chi$ as a corrector (see \cite{Jhikov}) and $a_{hom} := \lim _{x\to +\infty}\frac{1}{x}\int_0^x \nu (y)a(y)T'(y)^2dy$ is an effective diffusivity in the almost periodic or random stationary ergodic frameworks. We were not able to push further this observation, and hope to be able to derive more accurate estimates in these frameworks in a future work. 

This illustrates why we need to introduce function $T$ in order, somehow, to quantify the fluctuations in the estimate created by the heterogeneity. 

The proof of Theorem \ref{thm:heatkernel} follows the same steps as in \cite{FabesStroock} (for $W=0$) and \cite{Norris} (for periodic coefficients). However, we need to adapt these proofs in order to take into account the heterogeneity through the corrector $T$, and to check that this corrector satisfies appropriate properties. These are the main difficulties in this paper. 

It would be interesting to extend these results to multi-dimensional frameworks. The main difficulty would be the introduction of an appropriate corrector $T$. The methods we use in our proof to derive the properties of $T$ are one-dimensional ones. We thus leave this for future works. 

\bigskip

As a byproduct of our result of independent interest, we derive, as in \cite{FabesStroock}, a Nash estimate for the solutions of the canonical equation.

\begin{thm}\label{thm:Nash} Assume (\ref{hyp:original}) and let $p$  be   a solution of (\ref{eq:canonical}), with $p(0,\cdot)\in L^1 (\R)$. 
For each $\delta\in (0,1)$, there exist $C=C(\mu,\delta)$ and $\beta=\beta(\mu,\delta)\in (0,1)$ such that for all $R>0$ and $(s,\xi)\in (R^2,\infty) \times \R$:
$$|p(t,x)-p(t',x')|\leq C \|p\|_{L^\infty ((s-R^2,s)\times B(\xi,R))}\Big( \frac{\max \{\sqrt{|t-t'|}, |T(x)-T(x')-W(t-t')|\}}{R}\Big)^\beta$$
for all $(t,x), (t',x')\in (0,\infty)\times \R$, such that  
$$|T(x)-T(\xi)-Wt|<(1-\delta)R,  \quad |T(x')-T(\xi)-Wt'|<(1-\delta)R, $$
$$\hbox{ and } \quad s-(1-\delta^2)R^2<t,t'\leq s.$$

In particular, one has (for another constant depending on $\mu$ and $\delta$ that we still denote $C$), for all $t>0$, $x,x'\in \R$:
\begin{equation} \label{eq:osc}|p(t,x)-p(t,x')|\leq \frac{C}{t^{\frac{1+\beta}{2}}}\|p(0,\cdot)\|_{L^1 (\R)} |x-x'|^\beta.\end{equation}
\end{thm}

We could also derive the following $L^1-L^\infty$ continuity estimate (for which we do not provide a proof since it is immediate from Theorem \ref{thm:heatkernel}).

\begin{prop}\label{prop:lplq}Assume (\ref{hyp:original}) and let $p, q$  be   two solutions of (\ref{eq:canonical}), with $$p(0,\cdot)-q(0,\cdot)\in L^1 (\R).$$
Then there exists a constant $C=C(\mu)$ (independent of $W$) such that for all $t>0$:
$$\|p(t,\cdot)-q(t,\cdot)\|_{L^\infty (\R)}\leq \frac{C}{\sqrt{t}}\|p(0,\cdot)-q(0,\cdot)\|_{L^1 (\R)}.$$
\end{prop}


\subsection{Statement of the result for the original equation}\label{sec:statementoriginal}

Before stating our result for the original equation, we need to introduce the eigenelements that will unable us to reduce to a canonical equation. 
Let
\begin{equation} \underline{\gamma} := \sup_{\varphi \in H^{1}(\R)}\displaystyle \frac{\int_{\R}\big( r(x)\varphi^{2}-a(x)(\varphi')^{2}\big)dx}{\int_{\R}\varphi^{2}dx}. \end{equation}

The bounds (\ref{hyp}) on $a$ and $r$ ensure that $ \underline{\gamma}$ is well-defined and finite.

For $\gamma > \underline{\gamma}$, we know from \cite{Freidlin2, Noleninv} that
\be \label{eq:phigamma}
\big(a(x)\phi_{x}\big)_x + (r(x) - \gamma) \phi = 0, \quad x \in \R, \quad \phi>0, \quad \phi (0)=1, \quad \lim_{x\to +\infty}\phi (x)=0 
\ee
admits a unique solution $\phi=\phi_{\gamma}$. The existence and uniqueness of $\phi_\gamma$ relies on one-dimensional arguments and it is not clear how to define such a solution in $\R^N$. We define similarly a unique solution $\tilde{\phi}_{\gamma}$ with $\lim_{x\to -\infty}\tilde{\phi}_{\gamma} (x)=0 $ instead of $\lim_{x\to +\infty}\phi_{\gamma}(x)=0$.

\bigskip

The main result of this paper is the following.

\begin{thm}\label{thm:princ} Assume (\ref{hyp}) and let $\gamma>\underline{\gamma}$. Let $U(t,x,y)$  be   the fundamental solution of (\ref{eq:parab}), associated with the initial datum $\delta_y$.  Then there exists two constants $C>0$ and $W\in \R$ (depending on $\gamma$) and a function $T_\gamma=T_\gamma (x)=-W\dot{\phi}_\gamma (x)/\phi_\gamma (x)$, where $\dot{\phi}_\gamma$ is defined in Lemma \ref{lem:wgamma}, such that for all $t>0, x,y\in \R$:
$$
\displaystyle\frac{1}{Ct^{1/2}}\frac{\phi_\gamma (x)}{\phi_\gamma (y)} e^{-C|T_\gamma(x)-T_\gamma(y)-Wt|^{2}/t+\gamma t} \leq U(t,x,y)\leq \frac{C}{t^{1/2}}\frac{\phi_\gamma (x)}{\phi_\gamma (y)} e^{-|T_\gamma(x)-T_\gamma(y)-Wt|^{2}/Ct+\gamma t}. $$
\end{thm}

Indeed, the $\dot{\phi}_\gamma$ refers to the derivative of $\phi_\gamma$with respect to $\gamma$. 

When $a\equiv 1$ and $r\equiv 0$, one has $\phi_\gamma (x)=e^{-\sqrt{\gamma}x}$, $W=2\sqrt{\gamma}$, and $T_\gamma (x) = x$. Hence, 
$U(t,x,y)\equiv \frac{1}{\sqrt{4\pi t}}e^{-|x-y|^{2}/4t}$ could also be written 
$$U(t,x,y)\equiv  \frac{1}{\sqrt{4\pi t}} e^{-\sqrt{\gamma}(x-y)-\frac{1}{4t}|x-y-2\sqrt{\gamma}t|^{2}+\gamma t}= \frac{1}{\sqrt{4\pi t}}\frac{\phi_\gamma (x)}{\phi_\gamma (y)} e^{-|T_\gamma(x)-T_\gamma(y)-W t|^{2}/4t+\gamma t}$$
which is consistent with Theorem \ref{thm:princ} (even if our result is less accurate in this case since we need to introduce a constant $C>0$). In this case the estimate does not depend on $\gamma$ after simplification. 

When $a$ and $r$ are periodic with respect to $x$, then one can prove that $x-T_\gamma (x)$ stays bounded with respect to $x$. Hence, we recover Norris' result \cite{Norris} in that case. 

It is not possible in general to replace $T_\gamma (x)$ by $x$. If $r=r(x,\omega)$ is random stationary ergodic with respect to $(x,\omega)$, then we expect the fluctuations of $T_\gamma (x)$ around $x$ to be of order $\sqrt{x\ln\ln x}$. We leave this particular case for a future work. 

When $r\equiv 0$, one can prove that $\underline{\gamma}=0$, that $\phi_\gamma (x)\to 1$ as $\gamma\to \underline{\gamma}$ locally in $x$ and that $T_\gamma (x)$ converges to the constant $x$ as $\gamma\to \underline{\gamma}$. Hence, we could recover Aronson's \cite{Aronson} original result as well.
One needs to be careful however since the constant $C$ also depends on $\gamma$. 

We have one degree of freedom in this estimate, which is $\gamma>\underline{\gamma}$. It would thus be tempting to try to optimize this inequality with respect to $\gamma$. Bu the reader should keep in mind that $C$, $W$ and $T_\gamma$ depend on $\gamma$, and that it may happen that $C\to +\infty$ as $\gamma\to \underline{\gamma}$ or $+\infty$. We were thus unable to carry out such an optimization. One should thus choose $\gamma$ depending on the type of behavior of $U$ one wants to quantify.


\subsection{Estimates for Green functions of the canonical elliptic equations}\label{sec:statementGreen}

We now consider the Green solution $G_\lambda=G_\lambda(x,y)$ solutions of the canonical elliptic equation
\begin{equation} \label{eq:ellW}
-\big( a(x)\nu (x) G_x\big)_x +W G_x+\lambda W^2 \nu (x) G = \nu (x)\delta_y  \hbox{ in } \R,
\end{equation}
for some $\lambda>0$. 

\begin{prop}
Assume (\ref{hyp:original}). Then there exists a constant $C=C(\mu)>0$ independent of $W$, such that 
for all $x,y\in \R$, 
$$ \frac{1}{W C\sqrt{\lambda+C}} e^{-2W\sqrt{C}\big( \sqrt{\lambda +C}-\sqrt{C}\big)\big( T(x)-T(y)\big)}\leq G_\lambda(x,y)\leq 
 \frac{C}{W \sqrt{\lambda C+1}} e^{-\frac{2W}{C}\big( \sqrt{\lambda C+1}-1\big)\big( T(x)-T(y)\big)}.$$
\end{prop}

This immediately follows from the classical identity $G_\lambda (x,y)=\int_0^\infty e^{-\lambda W^2 t}P(t,x,y)dt$ and the following computation, available for all $a,b>0$ and $X\in \R$:
$$\int_0^\infty e^{-a t}\frac{e^{-b|X-t|^2/t}}{\sqrt{t}}dt=\sqrt{\frac{\pi}{a+b}}e^{-2\sqrt{b}(\sqrt{a+b}-\sqrt{b})X}.$$
See for example \cite{Bages} for a proof of this identity. 
%
%
%

\subsubsection*{Acknowledgements}
G. N. was partially supported by the ANR project ReaCh ({\it R\'eaction-diffusion : nouveaux
Challenges}).


\section{Properties of the function $T$}

In this section we first prove the existence and the properties of $T$ and of an adjoint function $\tilde{T}$. Then we introduce a flow $X$ associated with $T$ and the particular solution $f(t,x)=T(x)-Wt$ of the canonical equation, that will be a crucial tool in the proof of Theorem \ref{thm:heatkernel}. 

In the rest of the paper, we will assume that $W\geq 0$. The case $W\leq 0$ could be addressed with the change of variable $x\mapsto -x$.

\subsection{Proof of Proposition \ref{prop:T}.}

\begin{proof}[Proof of Proposition \ref{prop:T}.]
We define $T_R$ the unique solution of 
$$-\big( a(x)\nu(x)T_R'\big)'+WT_R'=\nu(x)W \hbox{ in } (0,R), \quad T_R(0)=0, \ T_R(R)=0.$$
If $W=0$, then $T_R\equiv 0$. If $W>0$, then one can easily prove that $T_R>0$ over $(0,R)$ and it follows that $R\mapsto T_R(x)$ is increasing for all $x>0$ (and $R>x$). Moreover, $T_R$ is uniformly bounded in $W^{2,\infty}_{loc}$ by elliptic regularity estimates. We could thus define $T(x):=\lim_{R\to +\infty} T_R(x)$. 

Let $\overline{x}\in [0,R]$ such that $(a\nu T_R')(\overline{x})=\max_{[0,R]}a\nu T_R'$. We want to prove that that $(\nu T_R')(\overline{x})\leq \mu^3$. If $(\nu T_R')(\overline{x})\leq 0$ we are done. Assume that $(\nu T_R')(\overline{x})>0$. We recall that $\mu\geq a,\nu \geq 1/\mu$. 
As $T_R>0$ over $(0,R)$ and $T_R(R)=0$, one has $T_R'(R)\leq 0$ and thus $\overline{x}<R$. For $\e>0$ small enough, one thus gets $(a\nu T_R')(\overline{x}+\e)\leq (a\nu T_R')(\overline{x})$ and thus $(a\nu T_R')'(\overline{x})\leq 0$. This leads to 
$$WT_R'(\overline{x})\leq -\big( a \nu T_R'\big)'(\overline{x})+WT_R'(\overline{x})=\nu(\overline{x})W.$$
Hence, $T_R'(\overline{x})\leq \nu(\overline{x})\leq \mu$. We conclude that $\max_{[0,R]}a\nu T_R'\leq \mu^3$ and thus $\max_{[0,R]} T_R'\leq \mu^5$, from which $T'\leq \mu^5$ follows by letting $R\to +\infty$. The inequality $T'\geq 1/\mu^5$ is proved similarly. 

Assume now that $S$ is another solution of (\ref{eq:T}) such that $S(x)/|x|$ is bounded. Let $R=S-T$. One has 
$$-\big( a(x)\nu(x)R'\big)'+WR'=0 \hbox{ in } \R, \quad R(0)=0.$$
Assume by contradiction that there exists $x_0$ such that $R'(x_0)>0$. Then one gets 
$$(a\nu R')(x)=(a\nu R')(x_0)e^{\int_{x_0}^x W/(a\nu)}\geq (a\nu R')(x_0)e^{W/\mu^2 (x-x_0)}.$$
This is a contradiction since $R'$ would then grow exponentially, contradicting $R(x)/|x|$ bounded. Hence $R'\leq 0$. One gets $R'\geq 0$ by symmetry and thus $R$ is constant equal to $0$ since $R(0)=0$. This shows uniqueness: $S\equiv T$. 
\end{proof}

We could similarly construct an adjoint solution. 

\begin{prop}
There exists a unique solution $\tilde{T}$ of 
\begin{equation} \label{eq:Ttilde}
\big( a(x)\nu(x)\tilde{T}'\big)'+W\tilde{T}'=\nu(x)W \hbox{ in } \R, \quad \tilde{T}(0)=0
\end{equation}
such that $\tilde{T}(x)/|x|$ is bounded over $\R$. Moreover, one has $m\leq \tilde{T}'(x)\leq M$ for all $x\in \R$, for $m=1/ \mu^5$ and $M= \mu^5$. 
\end{prop}

\begin{proof}
We just apply the change of variables $S(x):=-\tilde{T}(-x)$ and use Proposition \ref{prop:T}.
\end{proof}

\subsection{Definition and properties of the flow $X(t;y)$}

It will sometimes be more convenient, in particular when proving the lower bound in Theorem \ref{thm:princ}, to use the flow $X(t,y)$ instead of the time $T$, which is the inverse of $x\mapsto T (y)-T (x)$. 

Namely, let $X(t;y)$  be   the unique (since $m\leq T'\leq M$) solution of 
\begin{equation} \label{eq:defX}T\big(X(t;y)\big)-Wt=T(y)\quad \hbox{ for all } t,y\in \R.\end{equation} 

\begin{lem}\label{lem:Xsemigroup}
The function $X$ satisfies the semi-group property, in the sense that for all $s>0, t>0$ and $y\in \R$, one has 
$$X\big(s;X(t;y)\big) = X(t+s;y).$$
\end{lem}

\begin{proof}
One has 
$$T\Big(X\big(s,X(t;y)\big)\Big) -Wt-Ws=T\big(X(t;y)\big) -Wt=T(y).$$
The conclusion follows. 
\end{proof}

\begin{lem}\label{lem:f} The function 
$$f(t,x):= T (x)-Wt$$
is a time-global solution of 
\begin{equation} \label{eq:f}
\nu (x) \partial_{t}f - \partial_{x }\big( \nu (x)a(x)\partial_{x}f\big)+W\partial_{x}f=0 \hbox{ for all } t\in\R, x\in \R.
\end{equation}
Moreover, 
$f\big(t,\cdot+X(t,y)\big)-f(0,y)$ admits as a unique root $X(t;y)$ for all $t,x\in \R$, and one has:
\begin{equation} \label{eq:estg} \forall t>0, x,y\in \R, \quad m |x|\leq |f\big(t,x+X(t;y)\big)-f\big(0,y\big)| \leq M|x|, \end{equation}
and $W/M\leq X'\leq W/m$ over $\R$. 
\end{lem}

\begin{proof} 
One easily verifies that $f$ satisfies (\ref{eq:f}) and the fact that $X(t;y)$ is the unique root of $f\big(t,\cdot+X(t,y)\big)-f(0,y)$. 
Estimate (\ref{eq:estg}) follows from $m\leq T'\leq M$. 
By translating the origin one can assume that $y=0$. As $f\big( t,X(t;0)\big)=0$, one has $f_{t}+X'(t;0)f_{x}=0$ at $(t,X(t;0))$ and thus, as $f_{t}\equiv -W$ and $-M\leq f_{x}(t,X(t;0))\leq -m$, one has $W/M<X'(t;0)\leq W/m$. 
\end{proof}

Similarly,
let $Y(t;y)$  be   the unique solution of $\tilde{T}(Y(t;y))-Wt=\tilde{T}(y)$ for all $t,y\in R$. 
We could prove that 
$$g(t,x):= \tilde{T}(x)-Wt$$
is a time-global solution of 
\begin{equation} \label{eq:g}\left\{ \begin{array}{l}
-\nu (x) \partial_{t}g - \partial_{x }\big( \nu (x)a(x)\partial_{x}g\big)-W\partial_{x}g=0 \hbox{ for all } t\in\R, x\in \R,\\
g(t,Y(t;0))=0 \quad \hbox{ for all } t\in \R.\\ \end{array} \right.
\end{equation}
Moreover,
$$\forall t>0, x,y\in \R, \quad m |y|\leq |g\big(t,y+Y(t;x)\big)-g\big(0,x\big)| \leq M|y|.$$

\begin{lem}\label{lem:compfg} If $W\neq 0$,
there exists a constant $\tau=4\mu^7 /W$ such that 
\begin{equation} \label{eq:compfg} |T-\tilde{T}|\leq \tau.\end{equation}
Hence, 
\begin{equation} \label{eq:compXY} |X-Y|\leq 2\tau/m.\end{equation}
\end{lem}

\begin{proof}
Let $R:= T-\tilde{T}$. One has 
$$\big( a\nu(T+\tilde{T})'\big)'=W R'\hbox{ in } \R, \quad R(0)=0.$$
Integrating, one gets
$$W R(x)=\big(a\nu(T+\tilde{T})'\big)(x)-\big(a\nu(T+\tilde{T})'\big)(0) \leq 2\mu^2 (M-m)\leq 4\mu^7.$$
On the other hand, $W R\geq \frac{2}{\mu^2} (m-M)\geq -4\mu^3$. 

Next, as $T\big(X(t;y)\big)-\tilde{T}\big(Y(t;y)\big)=T(y)-\tilde{T}(y)$, one has 
$$m| X(t;y)-Y(t;y)|\leq |T\big(X(t;y)\big)-T\big(Y(t;y)\big)|\leq |(T-\tilde{T})(y)|+|(T-\tilde{T})\big(Y(t;y)\big)|\leq 2\tau.$$
\end{proof}

Lastly, we have the following technical inequality.

\begin{lem}\label{lem:compfg2}
There exists a constant $C=C(\mu)>0$ (independent of $W$) such that for all $t>0$, $x,y\in \R$:
$$|f(t,x)-f(0,y)|\leq C|g(t,x)-g(0,y)|+C\sqrt{t}.$$
\end{lem}

\begin{proof}
One computes
$$\begin{array}{rcl}
 |f\big(t,x\big)-f(0,y)|&\leq&M|x-X(t;y)|\\
 &\leq & M|x-Y(t;y)|+M|Y(t;y)-X(t;y)|\\
&\leq & \frac{M}{m}|g(t,x)-g(0,y)| +M|Y(t;y)-X(t;y)|.\\
\end{array}$$
Now, (\ref{eq:compXY}) yields $|X-Y|\leq 2\tau/m$ and Lemma \ref{lem:f} gives $|X(t;y)-Y(t;y)|\leq 2Wt/m$ for all $t>0$ and $y\in \R$. Hence, 
$|X(t;y)-Y(t;y)|\leq \frac{2}{m}\sqrt{W t\tau}$ for all $t>0$ and $y\in \R$. The conclusion follows since $W\tau=4\mu^7$ does not depend on $\mu$. 

\end{proof}

%

\section{The upper bound}

We define for all $\alpha>0$:
$$G_{\alpha}(t,x):=e^{\alpha  g(t,x)}.$$
Easy computations yield
\begin{equation} 
-\nu (x) \partial_{t}G_{\alpha} - \partial_{x }\big( \nu (x)a(x)\partial_{x}G_{\alpha}\big)-W\partial_{x}G_{\alpha}= -\alpha^{2}\nu (x) a(x)(\partial_{x}g)^{2}G_{\alpha} \hbox{ for all } t\in\R, x\in \R.\end{equation}

%

\begin{proof}[Proof of the upper bound in Theorem \ref{thm:heatkernel}]
Take $\alpha \in \R$ and let $V_{k}(t):= \int_{\R}\nu (x)G_{\alpha k}(t,x) p^{k}(t,x)dx$. Easy computations yield
$$\begin{array}{rl}V_{2k}'(t) =& 4\alpha^{2} k^2\int_{\R}\nu (x)a(x)g_{x}^{2}(t,x)G_{2\alpha k}(t,x) p^{2k}(t,x)dx\\
& - 2k(2k-1) \int_{\R}\nu (x)a(x)G_{2\alpha k}(t,x) p_{x}^{2}(t,x)p^{2k-2}(t,x)dx.\end{array}
$$
Let $\Psi:=G_{\alpha k}(t,x) p^{k}(t,x)$, so that $\Psi_x=\alpha k g_x G_{\alpha k} p^{k}+ k G_{\alpha k}p_x p^{k-1}$ and thus 
$$G_{2\alpha k}(p_x)^2 p^{2k-2}=\Big(G_{\alpha k}p_x p^{k-1}\Big)^2=\frac{1}{k^2}\Big( \Psi_x-\alpha k g_x G_{\alpha k} p^{k} \Big)^2.$$ 
We get 
$$\begin{array}{rcl}V_{2k}'(t) &=& 4\alpha^{2} k^2\int_{\R}\nu a g_{x}^{2}G_{2\alpha k}p^{2k} - \frac{2(2k-1)}{k} \int_{\R}\nu a\Big( \Psi_x^2-2\alpha k g_x G_{\alpha k} p^{k} \Psi_x+\alpha^2 k^2 g_x^2 G_{2\alpha k} p^{2k}\Big)\\
&\leq & 16\alpha^{2} k^2\int_{\R}\nu a g_{x}^{2}G_{2\alpha k}p^{2k} - \frac{(2k-1)}{k} \int_{\R}\nu a \Psi_x^2 \\
\end{array}
$$

Next, the Nash inequality applied to $\Psi$ yields that there exists a constant $C>0$ such that:
\begin{equation} \label{eq:V2k} V_{2k}'(t)\leq 16\alpha^{2}k^2 \|a\|_\infty M^{2}V_{2k}(t) - 2CV_{2k}^{3}(t)/V_{k}^{4}(t).\end{equation}

We now use the same arguments as in Section 1 of \cite{FabesStroock}. Namely, let $p_{j}:=2^{j}$, $U_{j}(t):= \| G_{\alpha}(t,\cdot)p(t,\cdot)\|_{L^{p_{j}}(\nu (x) dx)}$ and 
 $W_{j}(t):= \max \{ s^{(1/4-1/2p_{j})}U_{j}(s), \ 0\leq s\leq t\}$. 
 Then 
$$U_{j}'(t)\leq 8 \times 2^{j} \alpha^{2}M^{2}  \|a\|_\infty U_{j}(t)-\frac{C}{2^{j}}\Big(\frac{t^{1/4-1/2^{j}}}{W_{j-1}(t)}\Big)^{2^{j+1}}U_{j}(t)^{1+2^{j+1}}$$
and thus one derives from Lemma 1.4 of \cite{FabesStroock} that there exists a constant, that we still denote $C$, such that
 $$W_{j}(t)\leq W_{j-1}(t) (4^{j}C)^{1/2^{j+1}}e^{C\alpha^{2}t/2^{j}}.$$
We thus conclude that, even if it means increasing $C$, 
$$\sup_{j}W_{j}(t)\leq C e^{C\alpha^{2}t} W_{1}(t).$$
Hence, 
$$U_{j}(t)=\| G_{\alpha}(t,\cdot)p(t,\cdot)\|_{L^{p_{j}}(\nu (x) dx)} \leq \frac{C}{t^{1/4-1/2^{j+1}}} e^{C\alpha^{2}t} W_1(t).$$
Letting $j\to +\infty$, it follows that
$$\| G_{\alpha}(t,\cdot)p(t,\cdot)\|_{\infty} \leq \frac{C}{t^{1/4}} e^{C\alpha^{2}t} W_1(t).$$
We also know that 
$$U_{1}'(t)\leq 16 \alpha^{2}M^{2}  \|a\|_\infty U_{1}(t)$$
and thus $U_1(t)\leq U_1(0)e^{16 \alpha^{2}M^{2}  \|a\|_\infty t}$. As $W_1(t)=\max \{ U_1(s), 0\leq s\leq t\}$ by definition of $W_1$, we have proved that $W_1(t)\leq U_1(0)e^{C\alpha^2 t}$ for some constant $C$ and thus, even if it means increasing $C$:
\begin{equation} \label{eq:Uj}\| G_{\alpha}(t,\cdot)p(t,\cdot)\|_{\infty} \leq \frac{C}{t^{1/4}} e^{C\alpha^{2}t}\| G_{\alpha}(0,\cdot)p(0,\cdot)\|_{L^{2}(\nu (x) dx)}  .\end{equation}

On the other hand, one easily checks that 
$$0\leq V_{1}'(t) = 4\alpha^{2} \int_{\R}\nu (x)a(x)g_{x}^{2}(t,x)G_{2\alpha k}(t,x) p(t,x)dx\leq 4\alpha^{2}M^2 \|a\|_\infty V_1(t). $$
It follows that $V_1$ is nondecreasing and $V_{1}(t)\leq e^{4\alpha^{2}M^2\|a\|_\infty t} V_{1}(0)$.

Now, using again (\ref{eq:V2k}) with $k=1$, we get
$$U_{1}'(t)\leq 16\alpha^{2}M^{2}  \|a\|_\infty U_{1}(t)-\frac{C}{2}\Big(\frac{1}{U_0(t)}\Big)^{4}U_{1}(t)^{5}.$$
As $U_0\equiv V_1$ and $V_1$ is nondecreasing, Lemma 1.4 of \cite{FabesStroock} yields:
\begin{equation} \label{eq:Uj2}U_{1}(t)=\| G_{\alpha}(t,\cdot)p(t,\cdot)\|_{L^{2}(\nu (x) dx)} \leq \frac{C}{t^{1/4}} e^{C\alpha^{2}t} U_0(t)=\frac{C}{t^{1/4}} e^{C\alpha^{2}t} \| G_{\alpha}(0,\cdot)p(0,\cdot)\|_{L^{1}(\nu (x) dx)} .\end{equation} 

Combining (\ref{eq:Uj}) and (\ref{eq:Uj2}) thanks to the semi-group property, we eventually obtain
$$\| G_{\alpha}(t,\cdot)p(t,\cdot)\|_{\infty} \leq \frac{C}{t^{1/2} }e^{C\alpha^{2}t}  \| G_{\alpha}(0,\cdot)p(0,\cdot)\|_{L^{1}(\nu (x) dx)},$$
that is, for all $t>0, x\in \R$:
$$e^{\alpha g(t,x)}p(t,x)\leq  \frac{C}{t^{1/2}} e^{C\alpha^{2}t}  \int_{\R}\nu (y)e^{\alpha g(0,y)} p_{0}(y)dy.$$
It terms of the gaussian $P(t,x,y)$ associated with the initial datum $\delta_{y}$, this reads
$$P(t,x,y)\leq  \frac{C}{t^{1/2}} e^{C\alpha^{2}t-\alpha g(t,x)+\alpha g(0,y)}.$$
For any $t>0$, $x,y\in \R$, we now take $\alpha = \frac{g(t,x)-g(0,y)}{2Ct}$, which yields 
$$P(t,x,y)\leq  \frac{C}{t^{1/2}} e^{ -\frac{|g(t,x)-g(0,y)|^{2}}{4Ct}}.$$
It now follows from Lemma \ref{lem:compfg2} that, for a generic constant $C=C(\mu)>0$:
$$P(t,x,y)\leq  \frac{C}{t^{1/2}} e^{ -\frac{|f(t,x)-f(0,y)|^{2}}{4Ct}}.$$

This proves the upper bound in Theorem \ref{thm:heatkernel}.

\end{proof}

\section{The lower bound}

%
%
%

Define for all $t\in (0,1)$:
$$\rho (t,x'):= e^{-|\tilde{T}(x')-(t-1)W|^{2}/\sigma(2-t)},$$
with $\sigma= 4M^2$. 

\begin{lem}\label{lem:rho}
One has:
$$-\nu (x') \partial_{t}\rho - \partial_{x' }\big( \nu (x')a(x')\partial_{x'}\rho\big)-W\partial_{x'}\rho\geq 0.$$
\end{lem}

\begin{proof}
One computes:
$$\begin{array}{l}
-\nu (x') \partial_{t}\rho - \partial_{x' }\big( \nu (x')a(x')\partial_{x'}\rho\big)-W\partial_{x'}\rho\\
\displaystyle
= \frac{\nu (x')a(x')|\tilde{T}(x')-(t-1)W|^2}{\sigma(2-t)^2}\rho-\frac{4 a(x')\nu (x')|\tilde{T}(x')-(t-1)W|^2 g_{x'}^2(t,x')}{\sigma^2(2-t)^2}\rho\\
\displaystyle
+\frac{2\nu (x')a(x')g_{x'}^2(t,x')}{\sigma(2-t)}\rho\\
\displaystyle
\geq \frac{2\nu (x')a(x')g_{x'}^2(t,x')}{\sigma(2-t)}\rho\geq 0 \quad \hbox{(since $\sigma= 4M^2$)}.\\
\end{array}$$

\end{proof}

Let 
$$G_{x}(t):= \int_{\R}\rho(t,x')\nu (x') \ln \Big( P\big(t,x',Y(-1;x)\big)\Big)dx'.$$

Define the auxiliary functions 
$$Q(t):= \int_{\R}\rho(t,x')\nu (x')dx'$$
and 
$$H_{x}(t):= G_{x}(t) - Q(t)\ln \big(\overline{C}/t^{1/2}\big)=  \int_{\R}\rho(t,x')\nu (x') \ln \big( \frac{ t^{1/2}P\big(t,x',Y(-1;x)\big)}{\overline{C}}\big)dx',$$
where $\overline{C}$ is given by the upper bound in Theorem \ref{thm:heatkernel}, which
 yields that $H_{x}(t)\leq 0$ for all $t>0$. 

\begin{prop}\label{prop:entropy}
For all $R>0$, there exists a positive constant $B_{R}$ such that for all $x\in (-R,R)$:
$$ G_{x}(1)\geq -B_{R}+Q(1) \ln  \overline{C}.$$
\end{prop}

\begin{proof}[Proof of Proposition \ref{prop:entropy}.]

Equivalently, we need to prove that $H_{x}(1)\geq -B_R$. 

We compute using Lemma \ref{lem:rho} (here $P$ is always considered at $\big(t,x',Y(-1;x)\big)$, $\rho$ at $(t,x')$, and $\nu$ at $x'$): 
$$\begin{array}{rl}
H_{x}'(t)\geq &\displaystyle\int_{\R}\big(-(a\nu\rho_{x'})_{x'}-W\rho_{x'}\big) \ln \big( \frac{ t^{1/2}P}{\overline{C}}\big)dx'+\int_{\R}\frac{\rho}{2t}\nu  dx'\\
&\displaystyle+\int_{\R}\frac{\rho}{P}\big((\nu a P_{x'})_{x'}-W P_{x'}\big)  dx'.\\
\end{array}$$
We could integrate by parts and get
$$\begin{array}{rl}
H_{x}'(t)\geq &\displaystyle\int_{\R}\big(-(a\nu\rho_{x'})_{x'}-W\rho_{x'}\big)  \ln \big( \frac{ t^{1/2}P}{\overline{C}}\big)dx'+\int_{\R}\frac{\rho}{2t}\nu  dx'\\
&\displaystyle-W\int_{\R}\rho \frac{ P_{x'}}{P} dx'\displaystyle-\int_{\R}\rho_{x'}\nu a \frac{P_{x'}}{P} dx'+\int_{\R} \rho\nu a\frac{P_{x'}^{2}}{P^{2}}dx'\\
&\\
= &\int_{\R}\frac{\rho}{2t}\nu  dx'+\int_{\R} \rho\nu a \frac{P^{2}_{x'}}{ P^{2}}dx'\displaystyle\\
&\\
\geq &\int_{\R} \rho\nu a \frac{P^{2}_{x'}}{ P^{2}}dx'.\\
\end{array}$$



We are left with the term involving $P_{x'}^{2}$. As $\nu$ and $a$ have positive infimum, we could use the spectral gap inequality after a change of variables $X=\tilde{T}(x')-W(1-t)$, and $u(t,X):= \ln P \big(t,x',Y(-1;x)\big)$:
$$\begin{array}{l}
\displaystyle\int_{\R} \rho\nu a \frac{P^{2}_{x'}}{P^{2}}dx'
\geq \frac{1}{C} \int_\R e^{-|\tilde{T}(x')-W(t-1)|^{2}/C(2-t)}\frac{P^{2}_{x'}}{P^{2}}(t,x',Y(-1;x))dx'\\
\\
=\displaystyle \frac{1}{C}\int_\R e^{-|X|^{2}/C(2-t)}\frac{P^{2}_{x'}}{P^{2}}\big(t,\tilde{T}^{-1}(X+W(1-t)),Y(-1;x)\big)\frac{dX}{\tilde{T}'\big( \tilde{T}^{-1}(X+W(1-t))\big)}\\
\\
=\displaystyle \frac{1}{C}\int_\R e^{-|X|^{2}/C(2-t)} u_{X}^2\big(t,X\big)\tilde{T}'\big( \tilde{T}^{-1}(X+W(1-t))\big)dX\\
\\
\geq \displaystyle \frac{m}{C}\int_\R e^{-|X|^{2}/C(2-t)} u_{X}^2\big(t,X\big)dX \quad \hbox{ since } \tilde{T}'\geq m\\
\\
\geq \displaystyle \frac{2m}{C(2-t)}\int_\R e^{-|X|^{2}/C(2-t)}\Big(u(t,X)- \tilde{G}_{x'}(t)\Big)^{2}dX \quad \hbox{ by the spectral gap inequality}\\
\\
= \displaystyle \frac{2m}{C(2-t)}\int_\R \rho (t,x')\Big( \ln P\big(t,x',Y(-1;x)\big)- \tilde{G}_{x}(t)\Big)^{2} \tilde{T}'(x')dx'\\
\\
\geq  \displaystyle \frac{2m^2}{C(2-t)}\int_\R \rho (t,x')\Big( \ln P\big(t,x',Y(-1;x)\big)- \tilde{G}_{x}(t)\Big)^{2} dx' \hbox{ since } \tilde{T}'\geq m\\
\end{array}$$
where 
$$\begin{array}{rcl} \tilde{G}_{x}(t)&:=& \displaystyle\frac{1}{\sqrt{C\pi}}\int_\R\frac{e^{-|X|^{2}/C(2-t)}}{(2-t)^{1/2}} u(t,X)dX\\
&&\\
&= & \displaystyle\frac{1}{\sqrt{C\pi (2-t)}} \int_\R \rho (t,x')\ln P\big(t,x',Y(-1;x)\big)\tilde{T}'(x')dx'.\\
\end{array}$$
As $P(t,x',x) \leq \overline{C} /t^{1/2}$ by the upper bound in Theorem \ref{thm:heatkernel}, one gets, as $q\mapsto \big(\ln q - \tilde{G}_{x}(t)\big)^{2}/q$ is decreasing on $(e^{2+\tilde{G}_{x}(t)},\infty)$:
$$\begin{array}{l}
\displaystyle\int_{\R} \rho\nu a \frac{P^{2}_{x}}{P^{2}}dx
\geq \displaystyle \frac{2m^2}{C(2-t)}\frac{\big(\tilde{G}_{x}(t)-\ln (\overline{C} /t^{1/2})\big)^{2}}{\overline{C} /t^{1/2}}\int_{P\big(t,x',Y(-1;x)\big)\geq e^{2+\tilde{G}_{x}(t)}}\rho (t,x')  P(t,x',x)dx'\\
\\
\geq \frac{1}{C}  \big( \tilde{G}_{x}(t)-\ln (\overline{C} /t^{1/2})\big)^{2}\int_{P\big(t,x',Y(-1;x)\big)\geq e^{2+\tilde{G}_{x}(t)}}\nu(x') \rho(t,x')P(t,x',x)dx'\\
\end{array}$$
for all $t\in [1/2,1]$, for some new constant depending on $\mu$, that we still denote $C$. 

We now notice that 
$$\begin{array}{rcl}
 \tilde{G}_{x}(t)-\ln (\overline{C} /t^{1/2})&=& \displaystyle\frac{1}{\sqrt{C\pi (2-t)}}   \int_{\R}\rho(t,x')\ln ( \displaystyle\frac{ t^{1/2}P\big(t,x',Y(-1;x)\big)}{\overline{C}}) \tilde{T}'(x')dx'\\
 &&\\
&\leq & C \displaystyle \int_{\R}\nu (x')\rho(t,x')\ln ( \frac{ t^{1/2}P\big(t,x',Y(-1;x)\big)}{\overline{C}})dx'\\
&&\\
&=&  CH_{x}(t)\leq 0,\\
\end{array}$$
for some generic constant $C>0$. Hence, 
$\big( \tilde{G}_{x}(t)-\ln (\overline{C} /t^{1/2})\big)^{2}\geq \big(CH_{x}(t)\big)^{2}$.

On the other hand,
we know from the upper bound of Theorem \ref{thm:heatkernel} that there exists $A>0$ (which depends on $R$) such that for all 
$$\int_{|\tilde{T}(x')-W(t-1)|> A}\nu (x') P\big(t,x',Y(-1;x)\big)dx' \leq 1/2 \quad \hbox{ for all } (t,x)\in [1/2,1]\times (-R,R).$$
Hence, as 
$$\int_{\R} \nu (x')P\big(t,x',Y(-1;x)\big)dx' =\int_{\R} \nu (x')P\big(0,x',Y(-1;x)\big)dx'=1$$
since $P(0,\cdot,Y(-1;x))=\delta_{Y(-1;x)}/\nu$, one gets $\int_{|\tilde{T}(x')-W(t-1)|\leq A} \nu (x')P\big(t,x',Y(-1;x)\big)dx' \geq 1/2$ for all $(t,x)\in [1/2,1]\times (-R,R)$. 
This yields for all $t\in [1/2,1]$:
$$\begin{array}{l}
\int_{P(t,x',Y(-1;x))\geq e^{2+\tilde{G}_{x}(t)}}\nu(x') \rho(t,x') P\big(t,x',Y(-1;x)\big)dx'\\
\\
\geq  \int_{\R}\nu \rho P- e^{2+\tilde{G}_{x}(t)}\int_{\R}\nu \rho \\
\\
\geq e^{-A^{2}/\sigma}\int_{|\tilde{T}(x')-W(t-1)|\leq A}\nu  P- Ce^{2+\tilde{G}_{x}(t)}\\
\\
\geq\displaystyle \frac{1}{2}e^{-A^{2}/\sigma}- Ce^{2+CH_{x}(t)+\ln (\overline{C} /t^{1/2})}\\
\\
=\displaystyle\frac{1}{2} e^{-A^{2}/\sigma}- C e^{2+CH_{x}(t)}\\
\end{array}$$
for some generic constant $C>0$ depending on $\mu$, where we have used that $t\mapsto \int_{\R}\nu (x') \rho (t,x') dx'$ is bounded with respect to $t$. 

We conclude that there exists a constant $C>0$ such that 
$$\displaystyle\int_{\R} \rho\nu  \frac{P^{2}_{x}}{P^{2}}dx\geq \Big(\frac{1}{2} e^{-A^{2}/\sigma}- C e^{2+CH_{x}(t)}\Big)  \big( H_{x}(t)\big)^{2}.$$

It follows that for all $t\in [1/2,1]$:
$$H_{x}'(t)\geq   \Big( \frac{1}{2}e^{-A^{2}/\sigma}- C e^{2+CH_{x}(t)}\Big)  \big( H_{x}(t)\big)^{2}.$$

Assume that $C H_x(1)< -\frac{A^2}{\sigma}- \ln (4C)-2$. Then if there exists $t\in [1/2,1]$ such that 
$C H_x(t_0)= -\frac{A^2}{\sigma}- \ln (4C)-2$, taking the largest one, one would get for all $t\in [t_0,1]$:
$$H_{x}'(t)\geq  \frac{1}{4}e^{-A^{2}/\sigma}  \big( H_{x}(t)\big)^{2}>0,$$
and it would follow that $H_x$ would be increasing on $(t_0,1)$, contradicting  
$$C H_x(1)< -\frac{A^2}{\sigma}- \ln (4C)-2.$$
We have thus proved that $C H_x(t)< -\frac{A^2}{\sigma}- \ln (4C)-2$ for all  $t\in [1/2,1]$, from which it follows that 
$$H_{x}'(t)\geq  \frac{1}{4}e^{-A^{2}/\sigma}  \big( H_{x}(t)\big)^{2}>0 \quad \hbox{ on } [1/2,1].$$
Integrating on $(t,1)$ and using $H_x\leq 0$, this gives $H_x(1)\geq -8 e^{A^{2}/\sigma}$. 
We have thus proved that 
$$ H_x(1)\geq \min\big\{ -\frac{A^2}{C\sigma}- \frac{\ln (4C)}{C}-\frac{2}{C}, -8 e^{A^{2}/\sigma}\big\}.$$

%
%
%

\end{proof}

\begin{prop}\label{prop:lowbound}
For all $r>0$ large enough, there exists a constant $C$, which only depends on $\mu$, such that for all $t>0$ and $x,y\in \R$ such that 
$$|T(x)-T(y)-Wt|\leq r \sqrt{t},$$
one has 
$$P\big(t,x,y\big)\geq \frac{1}{C\sqrt{t}}.$$
\end{prop}

\begin{proof}[Proof of Proposition \ref{prop:lowbound}]
%
%
%
First of all, adapting \cite{FabesStroock}, a translation and scaling argument yields that it is enough to show that there exists a constant $C>0$, which only depends on $\mu$, not on $W$, such that for all $x,y\in (-R,R)$:
\begin{equation} \label{eq:scaling} P\big(2,X(1;x),Y(-1;y)\big)\geq \frac{1 }{C}.\end{equation}

Let us prove this claim. Assume that (\ref{eq:scaling}) is proved with a constant $C=C(\mu)>0$. Take $\sigma>0$, $z\in \R$, and let $P_\sigma^z (t,x',y'):= \sigma P\big(\sigma^2 t,\sigma (x'+z), \sigma (y'+z)\big)$. Then $P_\sigma$ satisfies
$$\nu (\sigma (x'+z)) \partial_{t}P_\sigma^z - \partial_{x }\big( \nu (\sigma (x'+z)) a(\sigma (x'+z)) \partial_{x}P^z_\sigma \big)+\sigma W\partial_{x}P^z_\sigma =0 \hbox{ for all } t\in (0,\infty), x\in \R.$$
It follows from (\ref{eq:scaling}) that, as the constant $C>0$ does not depend on $W$, one has for all $x,y\in (-R,R)$:
\begin{equation} \label{eq:Psigma}P_\sigma \big(2,X^z_\sigma (1,x'),Y^z_\sigma (-1,y')\big)\geq C,\end{equation}
where $X^z_\sigma$ is the unique solution of 
$$T^z_\sigma \big( X^z_\sigma (t,x')\big) -\sigma W t = T^z_\sigma (x')$$
and $T^z_\sigma$ is the unique solution $T$ of 
$$-\big( \nu (\sigma (x'+z))a(\sigma (x'+z))T'\big)'+ \sigma W T'=\sigma W \nu (\sigma (x'+z))\hbox{ in } \R, \quad T(-z)=-z$$
such that $x\mapsto T(x')/x'$ is bounded over $\R$ (see Corollary \ref{lem:Tgamma}). Hence, by uniqueness one has $T^z_\sigma (x') = T (\sigma (x'+z))/ \sigma-z$. It follows that $X^z_\sigma (t,x')=X (\sigma^2 t , \sigma (x'+z))/\sigma-z$. Similarly, $Y^z_\sigma (t,x')=Y (\sigma^2 t , \sigma (x'+z))/\sigma-z$.
Hence, we get from (\ref{eq:Psigma}):
\begin{equation} \label{eq:Psigma2}P \big(2\sigma^2,X (\sigma^2  , \sigma (x'+z)),Y(-\sigma^2 t,\sigma (y'+z)\big)\geq C/\sigma\end{equation}
for all $x',y'\in (-R,R)$, $\sigma>0$ and $z\in \R$.

Let $t>0$ and $x,y\in \R$ such that 
$$|T(x)-T(y)-Wt|\leq r \sqrt{t}.$$
Let $u, v\in \R$ such that $x=X(t/2,\sqrt{t/2}u)$ and $y=Y(-t/2,\sqrt{t/2}v)$. By definition of $X$ and $Y$, one has 
$$\begin{array}{rcl} 
m\sqrt{t/2}|u-v|&\leq &  |T(\sqrt{t/2}u)-T(\sqrt{t/2}v)| \\
&&\\
&\leq & |T(\sqrt{t/2}u)-T(\sqrt{t/2}v)|+|T(\sqrt{t/2}v)-\tilde{T}(\sqrt{t/2}v)|  \\
 &&\\
&=&|T(x)-\tilde{T}(y)-Wt|+|T(\sqrt{t/2}v)-\tilde{T}(\sqrt{t/2}v)|    \quad \hbox{ by def. of } X \hbox{ and } Y\\
 &&\\
&\leq&  (r+\sqrt{2}M)\sqrt{t} \quad \hbox{ since } |T'|\leq M \hbox{ and }  |\tilde{T}'|\leq M.\\
\end{array}$$
Hence, $|u-v|\leq  (r+\sqrt{2}M)\sqrt{2}/m$. 

We now take $\sigma =\sqrt{t/2}$, $z=\frac{1}{2}(u+v)$, $x'=\frac{1}{2}(u-v)$, $y'=\frac{1}{2}(v-u)$, and 
$R=(r+\sqrt{2}M)/\sqrt{2}m$. As $|x'|\leq R$ and $|y'|\leq R$, one gets from (\ref{eq:Psigma2}), using the definitions of $u$ and $v$, the result of Proposition \ref{prop:lowbound}. 

\bigskip

Let us now turn back to the proof of (\ref{eq:scaling}). 
Consider the adjoint fundamental solution $\hat{P}(t,x,y)$. One has the semi-group property:
$$P(2, x,y)= \int_{\R}\nu (y) P(1,x,z)P(1, z,y)dz.$$
Also, one easily checks that $P(t,x,z)=\hat{P}(-t,z,x)$.

Hence, 
$$P(2,x,y)= \int_{\R}\nu (z) \hat{P}(-1,z,x)P(1, z,y)dz.$$
It follows from the Jensen inequality that 
$$\begin{array}{l}\ln P\big(2,X(1;x),Y(-1;y)\big) \\
\geq  \ln \Big(\int_{\R}\nu (z) \hat{P}(-1,z,X(1;x))P(1, z,Y(-1;y)) e^{-|z|^{2}/\sigma} \Big)dz\\
\geq  \ln Q(1)+ \frac{1}{Q(1)} \int_{\R}\nu (z)e^{-|z|^{2}/\sigma}   \ln\big(\hat{P}(-1,z,X(1;x)) P(1, z,Y(-1;y)) \big) dz\\
= \ln Q(1)+ \frac{1}{Q(1)} \big( G_{y}(1)+\hat{G}_{x}(-1)\big)\\
\end{array}
$$
where
$$\hat{G}_{x}(t):= \int_{\R}\nu (z)e^{-|T(x)-W(1+t)|^{2}/\sigma (2+t)} \ln \big( \hat{P}(t,z,X(1;x))\big)dz.$$
It follows from Proposition \ref{prop:entropy} applied to $G_{y}$ and $\hat{G}_{x}$, that:
$$\ln P\big(2,X(1;x),Y(-1;y)\big) \geq \ln Q(1)+ \frac{1}{Q(1)}\Big( -2B_{R}+2Q(1) \ln \overline{C}\Big).$$
One could easily check from the definition of $Q$ that $Q(1)\geq 1/C$ for some constant $C>0$. The conclusion follows. 
\end{proof}

\begin{proof}[Proof of the lower bound in Theorem \ref{thm:heatkernel}]

We know from Proposition \ref{prop:lowbound} that there exists $C>0$ such that for all $x,y$ such that 
$|T(x)-T(y)-Wt|\leq 3\sqrt{t},$
one has
$P(t,x,y)\geq \frac{1}{C\sqrt{t}}$.

Consider $x,y\in \R$, $t>0$ and let $k\in \mathbb{N}$ such that 
$$|T(x)-T(y)-Wt|\leq \sqrt{kt}.$$
We choose points $x_0=y, x_1,...,x_{k-1}, x_k=x$ such that for all $i=0,...,k-1$:
$$|T(x_{i+1})-T(x_i)-Wt/k|\leq \sqrt{t/k}.$$

 The semi-group property yields for all $t>0$:
 $$ P\big(t,x,y\big)\geq  \int ...\int_{|x'_i-x_i|\leq \frac{1}{M}\sqrt{t/k}}\Pi_{i=0}^{k-1}P\big(t/k, x_{i+1}', x_{i}'\big) dx_{1}'... dx_{k-1}'.$$
 If $|x'_i-x_i|\leq \frac{1}{M}\sqrt{t/k}$, one has 
 $$|T(x_{i+1}')-T(x_i')-Wt/k|\leq |T(x_{i+1})-T(x_i)-Wt/k|+ |T(x_{i+1}')-T(x_{i+1})|+ |T(x_{i}')-T(x_{i})|\leq 3\sqrt{t/k}$$
 since $|T'|\leq M$, and thus $P(t/k,x_{i+1}',x_i')\geq \frac{1}{C\sqrt{t/k}}$.
 
  Hence,
 $$ P\big(t,x,y\big)\geq \Big(\frac{\sqrt{k}}{C\sqrt{t}}\Big)^{k}\Big( \frac{2}{M}\sqrt{t/k}\Big)^{k-1}= \frac{M}{2}(CM/2)^{-k}\sqrt{k/t}\geq e^{-Ck}/C\sqrt{t}$$
 for some alternative constant that we still denote $C$. As $|T(x)-T(y)-Wt|\leq \sqrt{kt}$, we thus conclude that 
 $$P\big(t,x,y\big)\geq \frac{e^{-C|T(x)-T(y)-Wt|^2/t}}{C\sqrt{t}}.$$

\end{proof}


\section{The Nash type estimate}

We now let 
$$T(\xi,R,s):=\{ (t,x)\in \R\times\R, \ |T(x)-T(\xi)-Wt|<R, \ t\geq s\}$$
and we 
consider the fundamental solution $P^{(\xi, R)}=P^{(\xi, R)}(t,s,x,y)$ associated with equation 
\begin{equation} \label{eq:pxi} \left\{\begin{array}{ll}
\nu(x)p_t- \big( a(x)\nu (x)p_x\big)_x- Wp_x=0 & \hbox{ in } T(\xi,R,s),\\ 
p(t,x)=0 &\hbox{ for all } (t,x)\in \partial T(\xi,R,s)\cap (s,\infty),\\
\end{array}\right.
\end{equation}
with initial datum $P(s,s,\cdot,y)=\delta_y/ \nu (y)$. 

\begin{lem}\label{lem:PxiR}
For each $\delta \in (0,1)$, there is an $\alpha=\alpha (\mu,\delta)>0$ such that 
$$P^{(\xi,R)}(t,s,x,y)\geq \frac{\alpha}{C\sqrt{t-s}}e^{-C\frac{|T(x)-T(y)-W(t-s)|^2}{t-s}}$$
for all $(t,x)\in T(\xi,\delta R,s)$, $(s,y)\in T(\xi,\delta R,s)$, and $t\in (s,s+R^2)$, where $C=C(\mu)$ is the same as in Theorem \ref{thm:heatkernel}. 
In particular, if one also has $t-s\geq \gamma R^2$, then $P^{(\xi,R)}(t,s,x,y)\geq \frac{\alpha}{C\sqrt{\gamma}R}e^{-4C\delta^2/\gamma}$. 
\end{lem}

\begin{proof}
By translation, we can assume that $\xi=0$ and $s=0$, and we denote $T(R):=T(0,R,0)$. There exist two nonnegative functions on $\R^+$, $mes_y^+$ and $mes_y^-$ with total mass less or equal to $1$, such that 
$$\begin{array}{ll}P^{(0,R)}(t,0,x,y)=&P(t,0,x,y)-\int_0^t P\big(t,r,x,X(R+Wr,0)\big)mes^+_y(r)dr\\
&-\int_0^t P\big(t,r,x,X(-R+Wr,0)\big)mes^-_y(r)dr,\\ \end{array}$$
where we remind to the reader that $z:=X(\pm R+Wr,0)$ is the unique solution of $T(z)=\pm R+Wr$. 
We refer to \cite{HNRR2} for a proof of this claim in the periodic framework, that is indeed still available in the general framework, with 
$$mes_y^+ (r) =\nu \big(X(R+Wr,0)\big) |P_x^{(0,R)}\big(r,0,X(R+Wr,0),y\big)|$$  $$mes_y^- (r) = \nu \big(X(-R+Wr,0)\big)|P_x^{(0,R)}\big(r,0,X(-R+Wr,0),y\big)|.$$ 

Hence, by Theorem \ref{thm:heatkernel}, one has 
$$P^{(0,R)}(t,0,x,y)\geq \frac{1}{C\sqrt{t}} e^{-C|T(x)-T(y)-Wt|^2/t} -C\sup_{0\leq \tau\leq t} \frac{1}{\sqrt{\tau}} e^{-R^2(1-\delta)^2/C\tau}$$
for $(t,x)\in T(\delta R)$. In particular, there is an $\e\in (0,1-\delta)$, depending only on $C$ and $\delta$, such that 
$$P^{(0,R)}(t,0,x,y)\geq \frac{1}{2C\sqrt{t}} e^{-C|T(x)-T(y)-Wt|^2/t} $$
for all $(t,x)\in T( \delta R)$ and $t\in (0,\e^2 R^2]$, and $y$ with $|T(x)-T(y)-Wt|<\e R$.

 We could conclude as in the derivation of the lower bound in the proof of Theorem \ref{thm:heatkernel} that 
 $$P^{(0,R)}(t,0,x,y)\geq \frac{\alpha}{C\sqrt{t}} e^{-C|T(x)-T(y)-Wt|^2/t} $$
 for some $\alpha>0$,
for all $(t,x)\in T(\xi,\delta R,s)$, $(s,y)\in T(\xi,\delta R,s)$, and $t \in (0,R^2)$. This concludes the proof.

\end{proof}

\begin{lem}\label{lem:osc}
For each $\delta>0$, let $\rho:= 1- \e$ where $\e$ is given by Lemma \ref{lem:PxiR}. Then for all $s,\xi\in \R$ and $R>0$:
$$Osc (p;s,\xi,\delta R)\leq \rho Osc (p;s,\xi, R)$$
for any solution $p$ of (\ref{eq:canonical}), where 
$$Osc (p;s,\xi, R):= \sup \{ |p(t,x)-p(t',x')|, (t,x), (t',x')\in T(\xi, R,s)\}.$$
\end{lem}

\begin{proof}
We need to adapt the proof of lemma 5.2 in \cite{FabesStroock}. Define 
$$S:=\Big\{ y\in \R, \ |T(y)-T(\xi)-W(s-R^2)|<\delta R \hbox{ and } p(s-R^2,y)>\frac{M(R)+m(R)}{2}\Big\}.$$
We can assume that 
$$|S|\geq \frac{1}{2} | \{ y\in \R, \ |T(y)-T(\xi)-W(s-R^2)|<\delta R\}|,$$
otherwise we consider $1-p$ instead of $p$.
For all $(t,x)\in T(\xi, \delta R, s-\delta^2 R^2)$, one has 
$$\begin{array}{rcl}
p(t,x)-m(R)& \geq & \int_\R \nu (y) \big( p(s-R^2,y)-m(R)\big) P^{(\xi,R)}(t,s-R^2,x,y)dy\\
&&\\
& \geq & \frac{M(R)-m(R)}{2}\int_S \nu (y)P^{(\xi,R)}(t,s-R^2,x,y)dy\\
&&\\
& \geq & \frac{M(R)-m(R)}{2}\int_S \nu (y) \frac{\alpha}{CR}e^{-4C\delta^2/(1-\delta^2)}dy \quad \hbox{using Lemma \ref{lem:PxiR}}\\
&&\\
& \geq & \frac{M(R)-m(R)}{4} \inf_\R\nu (y) \frac{\alpha}{C}e^{-4C\delta^2/(1-\delta^2)}dy=: \e \big(M(R)-m(R)\big)\\
\end{array}$$
where $\e$ is an arbitrary small constant only depending on $\mu$ and $\delta$. 
The conclusion follows with $\rho:=1-\e$. 
\end{proof}

\begin{proof}[Proof of Theorem \ref{thm:Nash}.]
The first part of the Theorem could be derived as Theorem 5.3 of \cite{FabesStroock} using Lemma \ref{lem:osc}. 

In order to derive (\ref{eq:osc}), we first notice that it follows from Theorem \ref{thm:heatkernel} that 
$$p(t,x)\leq \frac{C}{\sqrt{t}}\int_\R e^{-|T(x)-T(y)-Wt|^2/Ct}p(0,y)dy\leq \frac{C}{\sqrt{t}} \|p(0,\cdot)\|_{L^1 (\R)}.$$
Take  $t=t'=s$, $\xi=X(-t;x)$,$(1-\delta)=1/\sqrt{2}$ and $R=\sqrt{t}$. By definition of $X$, one has $T(x)-T(\xi)-Wt=0$ and 
$|T(x')-T(\xi)-Wt|=|T(x)-T(x')|$. Hence, if $|T(x)-T(x')|\leq \sqrt{t/2}=R1-\delta)$, one gets from the first part of the Theorem: 
$$|p(t,x)-p(t,x')|\leq C \|p\|_{L^\infty ((t/2,t)\times B(\xi,R))}\Big( \frac{|T(x)-T(x')|}{\sqrt{t/2}}\Big)^\beta.$$
If $|T(x)-T(x')|\geq \sqrt{t/2}$, this inequality still holds with $C=2$. Hence, for all $t>0$ and $x,x'\in \R$, one has
$$|p(t,x)-p(t,x')|\leq C \|p\|_{L^\infty ((t/2,t)\times B(\xi,R))}\Big( \frac{|T(x)-T(x')|}{\sqrt{t/2}}\Big)^\beta\leq 
\frac{C}{t^{\frac{1+\beta}{2}}}\|p(0,\cdot)\|_{L^1 (\R)} |x-x'|^\beta
$$
for some generic constant $C$ depending on $\mu$ and $\beta$. 
\end{proof}


\section{Proof of the estimates for the original equation}

\subsection{Definitions and properties of the Wronskian and the invariant measure}\label{sec:UP}

We first define the Wronskian, which is known to be constant. 
\begin{lem}\label{lem:Wronskian}
Let $W_{\gamma}:=a \tilde{\phi}_{\gamma}' \phi_{\gamma}- a\phi_{\gamma}' \tilde{\phi}_{\gamma}$. 
Then $W_{\gamma}$ is a positive constant over $\R$. 
\end{lem}

Then, if $U$ is the fundamental solution associated with (\ref{eq:parab}), easy computations yield that $P(t,x,y)=U(t,x,y)\phi_\gamma (y)/ \phi_\gamma (x) e^{\gamma t} $ is the fundamental solution associated with (\ref{eq:canonical}) with $W=W_\gamma$ and $\nu=\nu_\gamma:= \phi_{\gamma} \tilde{\phi}_{\gamma}$. In order to derive Theorem \ref{thm:princ} from Theorem \ref{thm:heatkernel}, we need to check that $\nu_\gamma$ and $W_\gamma$ satisfy hypothesis (\ref{hyp:original}).

\begin{prop}\label{prop:measure}
For $\gamma > \underline \gamma$, define $\nu_{\gamma}:= \phi_{\gamma} \tilde{\phi}_{\gamma}$. Then 
\begin{itemize}
\item  $\inf_{\R}\nu_{\gamma}>0$,
\item $\nu_{\gamma}$ is bounded. 
\end{itemize}
\end{prop}


The proof of this Proposition will rely on the following Lemma. 

\begin{lem}\label{lem:ineqphi}
For all $\gamma>\underline{\gamma}$, there exists $\e>0$ such that 
$$\forall x\in \R, \quad \frac{\tilde{\phi}_{\gamma}'(x)}{\tilde{\phi}_{\gamma}(x)} \geq \frac{\phi_{\gamma}'(x)}{\phi_{\gamma}(x)} + \e.$$
\end{lem}

\begin{proof}  Assume first that $x=0$. 
We know from Lemma 2.7 of \cite{Noleninv} that if $\gamma>\gamma'>\underline{\gamma}$, for all 
$0<\e \leq\big( \sqrt{\gamma -\inf_{\R}r}-\sqrt{ \gamma' - \inf_{\R}r}\big)/\sqrt{\inf_\R a}$:
$$\forall x\geq 0, \quad \phi_{\gamma}(x)\leq  \phi_{\gamma'}(x)e^{-\e x}.$$

The same result applies to $\tilde{\phi}_\gamma$ with the change of variable $x\mapsto -x$, yielding: 
$$\forall x\leq 0, \quad \tilde{\phi}_{\gamma}(x)\leq  \tilde{\phi}_{\gamma'}(x)e^{\e x}.$$
Hence, $ \tilde{\phi}_{\gamma}'(0)\geq  \tilde{\phi}_{\gamma'}'(0)+\e.$

Let now prove the following claim
$$\forall x\geq 0, \quad \tilde{\phi}_{\gamma}(x)\geq  \tilde{\phi}_{\gamma'}(x).$$
Define the Wronskian $Z(x)= \tilde{\phi}_{\gamma'}'(x)\tilde{\phi}_\gamma(x)- \tilde{\phi}_{\gamma'}(x)\tilde{\phi}_\gamma'(x).$ 
Then, $$(aZ)'(x)=(\gamma'-\gamma)  \tilde{\phi}_{\gamma'}(x)\tilde{\phi}_\gamma(x)<0.$$ As $Z(0)\leq -\e$, one has $Z(x)\leq 0$ for all $x\geq 0$, and thus $\tilde{\phi}_{\gamma'}'(x)/\tilde{\phi}_{\gamma'}(x)\leq \tilde{\phi}_\gamma'(x)/\tilde{\phi}_\gamma(x)$ for all $x\geq 0$, and the claim follows by integration from $0$ to $x\geq 0$.

Combining these two inequalities, one gets 
$$\forall x\geq 0, \quad \frac{\tilde{\phi}_{\gamma}(x)}{\phi_{\gamma}(x)}\geq  \frac{\tilde{\phi}_{\gamma'}(x)}{\phi_{\gamma'}(x)}e^{\e x}.$$
Moreover, we know that $\tilde{\phi}_{\gamma'}(x) \geq \phi_{\gamma'}(x)$ for all $x\geq 0$. Hence, 
$$\forall x\geq 0, \quad \ln\tilde{\phi}_{\gamma}(x)\geq \ln\phi_{\gamma}(x)+\e x.$$
Taking the first order Taylor development near $x=0^{+}$, one gets 
$$\frac{\tilde{\phi}_{\gamma}'(0)}{\tilde{\phi}_{\gamma}(0)} \geq \frac{\phi_{\gamma}'(0)}{\phi_{\gamma}(0)} + \e.$$
In order to handle the case $x\neq 0$, we just translate the origin and take $\phi_{\gamma}(\cdot)/\phi_{\gamma}(x)$.
\end{proof}

\begin{proof}[Proof of Proposition \ref{prop:measure}.]
We first notice that 
$$\frac{W_{\gamma}}{a\nu_{\gamma}} = \frac{\tilde{\phi}_{\gamma}'}{\tilde{\phi}_{\gamma}}-\frac{\phi_{\gamma}'}{\phi_{\gamma}}.$$
The Harnack inequality yields that $\frac{\phi_{\gamma}'}{\phi_{\gamma}}$ and $\frac{\tilde{\phi}_{\gamma}'}{\tilde{\phi}_{\gamma}}$ are bounded, so $\nu_{\gamma}$ admits a positive infimum. 

Lastly, Lemma \ref{lem:ineqphi} yields that 
$$\frac{W_{\gamma}}{a\nu_{\gamma}} \geq \e$$
and thus $\nu_{\gamma}$ is bounded. 
\end{proof} 

\subsection{Convexity of $\phi_\gamma$ with respect to $\gamma$}

\begin{lem}\label{lem:phiconvexe}
For all $\gamma, \gamma'>\underline{\gamma}$, one has for all $x\in \R$, $\sigma \in (0,1)$:
$$\frac{\phi_{(1-\sigma)\gamma+\sigma \gamma'}'(x)}{\phi_{(1-\sigma)\gamma+\sigma \gamma'}(x)}\leq (1-\sigma) \frac{\phi_{\gamma}'(x)}{\phi_{\gamma}(x)}+\sigma\frac{\phi_{\gamma'}'(x)}{\phi_{\gamma'(x)}}.$$
\end{lem}

\begin{proof}
We could always assume that $x=0$ by translation. 
Next,  classical arguments from Lemma 2.5 of \cite{Noleninv} give
$$\phi_{(1-\sigma)\gamma+\sigma \gamma'}(x)\leq \phi_{\gamma}(x)^{1-\sigma}\phi_{\gamma'}(x)^{\sigma} \quad \hbox{ for all } x\in \R.$$
Expending near $x=0^{+}$, one gets 
$$\phi_{(1-\sigma)\gamma+\sigma \gamma'}'(0)\leq (1-\sigma)\phi_{\gamma}'(0)+\sigma\phi_{\gamma'}'(0),$$
which ends the proof when $x=0$. 
\end{proof}

%


\subsection{The derivative $\dot{\phi}_\gamma$ and its properties}

\begin{lem}\label{lem:wgamma}
The function $\gamma\mapsto \phi_\gamma$ admits a derivative $\dot{\phi}_{\gamma}$ for all $\gamma>\underline{\gamma}$, which is the unique solution of
\begin{equation} \label{eq:phipoint}\big( a(x)\dot{\phi}_{\gamma}'\big)'+\big(r(x)-\gamma\big)\dot{\phi}_{\gamma}=\phi_{\gamma} \hbox{ in } \R, \quad \dot{\phi}_{\gamma}(0)=0\end{equation}
such that $x\mapsto \displaystyle\frac{\dot{\phi}_{\gamma}}{x\phi_{\gamma}}$ is bounded over $\R$. 
\end{lem}

\begin{proof}
First, the convexity of $\gamma\mapsto \ln \phi_{\gamma}(x)$ mentioned in the proof Lemma \ref{lem:phiconvexe} yields that one can always define a left derivative 
$$\dot{\phi}_{\gamma}(x):=\phi_{\gamma}(x)\times\lim_{\gamma'\to \gamma^{-}}\displaystyle\frac{\ln\phi_{\gamma}(x)-\ln\phi_{\gamma'}(x)}{\gamma-\gamma'}.$$

On the other hand, we know from Lemma \ref{lem:phiconvexe} that $\gamma \mapsto \phi_{\gamma}'(x)/\phi_{\gamma}(x)$ is convex for all $x\in \R$. One could easily check that 
$$|\phi_{\gamma}'(x)/\phi_{\gamma}(x)|\leq \sqrt{\displaystyle\frac{\gamma-\inf_{\R}r}{\inf_\R a}}.$$
Hence, it is a well-known property of convex functions that $\phi_{\gamma}'(x)/\phi_{\gamma}(x)$ is $\frac{\sqrt{\gamma+\delta-\inf_{\R}r}}{\delta\sqrt{\inf_\R a}}$-Lipschitz-continuous with respect to $\gamma$ on any ball of radius $\delta$. By exchanging the derivatives with respect to $\gamma$ and $x$, we get that  $x\mapsto \displaystyle\frac{\dot{\phi}_{\gamma}}{x\phi_{\gamma}}$ is bounded over $\R$. 

Lastly, if $\psi$ is another solution of (\ref{eq:phipoint}) such that $x\mapsto \displaystyle\frac{\psi}{x\phi_{\gamma}}$ is bounded over $\R$, then $z:=\dot{\phi}_{\gamma}-\psi$ would satisfy
$$\big(a(x)z'\big)'+\big( r(x)-\gamma\big)z=0 \hbox{ over } \R, \quad z(0)=0.$$
We could thus write it $z=A\phi_{\gamma}+B\tilde{\phi}_{\gamma}$ since these functions are two independent solutions of the equation. Moreover, $z(0)=A+B=0$ and thus $B=-A$. Dividing by $x\phi_{\gamma}(x)$, as $\tilde{\phi}_{\gamma}/\phi_{\gamma}$ converges at least exponentially to $+\infty$ as $x\to +\infty$ by Lemma \ref{lem:ineqphi}, one gets a contradiction unless $A=0$, which means that $z\equiv 0$. Hence $\dot{\phi}_{\gamma}$ is uniquely defined. 
\end{proof}

\begin{cor}\label{lem:Tgamma}
The function $T_\gamma:=-W_\gamma \dot{\phi}_{\gamma}/\phi_\gamma$ is the unique solution of
$$-\big( \nu (x)a(x)T_{\gamma}'\big)'+ W_\gamma T_{\gamma}'=W_\gamma \nu (x) \hbox{ in } \R, \quad T_{\gamma}(0)=0$$
such that $x\mapsto \displaystyle\frac{T_{\gamma}}{x}$ is bounded over $\R$. 
\end{cor}

\begin{proof} The existence follows from Lemma \ref{lem:wgamma} and easy computations. The uniqueness follows from Proposition \ref{prop:T}.\end{proof}


\subsection{Proof of Theorem \ref{thm:princ}} 

\begin{proof}[Proof of Theorem \ref{thm:princ}.]

Let $U$  be   the fundamental solution associated with (\ref{eq:parab}), easy computations yield that $U(t,x,y)=P(t,x,y)\phi_\gamma (y)/ \phi_\gamma (x) e^{\gamma t} $ is the fundamental solution associated with (\ref{eq:canonical}) with $W=W_\gamma$ and $\nu=\nu_\gamma:= \phi_{\gamma} \tilde{\phi}_{\gamma}$. Proposition \ref{prop:measure} yields $\nu_\gamma$ satisfies hypotheses \ref{hyp:original}. Corollary \ref{lem:Tgamma} yields that  $T_\gamma:=-W_\gamma \dot{\phi}_{\gamma}/\phi_\gamma$. The conclusion follows. 

\end{proof}


\section{Conclusion}

We conclude this paper by explaining what are the difficulties when trying to extend the results to multi-dimensional media. Let first consider the equation
\begin{equation} \label{eq:N}\left\{\begin{array}{ll}
u_{t}-\nabla \cdot \big(A(x)\nabla u\big)= r(x)u \quad& \hbox{ for all } t>0, \ x\in \R^N,\\
u(0,x)=u_{0}(x) \quad &\hbox{ for all } x\in \R^N,\\
\end{array}\right. \end{equation}
where $A(x)$ is a symmetric $N\times N$ matrix for a.e. $x\in \R^N$ and 
\begin{equation}\label{hyp}
\exists \mu>0, \ |r(x)|\leq \mu, \ \frac{1}{\mu}I_N\leq A(x)\leq \mu I_N  \hbox{ in the sense of symmetric matrix for a.e. } x\in \R. 
\end{equation}

It is known (see \cite{BR}) that for all $\gamma\geq \underline{\gamma}$, there exists a positive solution $\phi_\gamma$ of 
\begin{equation} \label{eq:phid} \nabla \cdot \big(A(x)\nabla \phi_\gamma \big)+ r(x)\phi_\gamma = \gamma \phi_\gamma \quad \hbox{ in } \R^N.\end{equation}
But it is not clear what is the behaviour of $\phi_\gamma$ near $|x|\to +\infty$ and what would be the equivalent of (\ref{eq:phigamma}). Only in some specific non-periodic situations, one could construct some $\phi_\gamma$ having a good prescribed behaviour near $+\infty$ (see for example \cite{corrector} that addresses the random stationary ergodic framework). 
Next, consider two positive solutions $\phi_\gamma$ and $\tilde{\phi}_\gamma$ of (\ref{eq:phid}),  and let us try to introduce some invariant measure $\nu_\gamma:=\phi_\gamma \tilde{\phi}_\gamma$ as in Proposition \ref{prop:measure} but, unless we have some good properties for $\phi_\gamma$, one cannot prove that this invariant measure is bounded. Hence, our one-dimensional proof does not work anymore. 

Next, consider the canonical equation 
\begin{equation} \label{eq:canonical}\left\{ \begin{array}{ll}
\nu (x) \partial_{t}p - \nabla \cdot \Big( \nu (x)\big(A(x)\big) \nabla p\big)+W(x)\cdot \nabla p=0 &\hbox{ for all } t\in (0,\infty), x\in \R^N,\\
p(0,x) = p_0(x) &\hbox{ for all } x\in \R^N,\\
\end{array}\right. \end{equation} 
where $A=A(x)$ is a symmetric $N\times N$ matrix for a.e. $x\in \R^N$, $\frac{1}{\mu}\leq \nu(x)\leq \mu$ and $W=W(x)$ is a divergence-free vector field in $\R^N$.

We could then construct a corrector $T: \R^N \to \R^N$ solution of 
$$- \nabla \cdot \Big( \nu (x)\big(A(x)\big) \nabla T\big)+\big(W(x)\cdot \nabla \big) T= \nu (x) W \quad \hbox{ in } \R^N$$
and an adjoint solution $\tilde{T}$. Most of the computations in the proof could be performed, however, we need two estimates that we did not manage to derive and that could possibly be false in general.
\begin{itemize}
\item In order to derive the upper bound, we need $|\nabla T|\leq M$ for some $M>0$ independent of $W$.
\item In order to derive the lower bound, we need $\mathrm{det} \nabla T\geq m$ for some $m>0$ independent of $W$.
\end{itemize}
Without these estimates, we are stuck and we cannot even derive weaker bounds. 

These are the reasons why we believe the multi-dimensional problem is much more difficult to address.

\end{document}